\documentclass[12pt]{article}

\usepackage{mathrsfs}

\usepackage{amsmath,amsthm,amssymb}
\usepackage{graphicx}
\usepackage{vector}

\font\tencmmib=cmmib10 \skewchar\tencmmib '60
\newfam\cmmibfam
\textfont\cmmibfam=\tencmmib

\def\lessim{\ \lower4pt\hbox{$
\buildrel{\displaystyle <}\over\sim$}\ }
\def\gessim{\ \lower4pt\hbox{$\buildrel{\displaystyle >}
\over\sim$}\ }

\newcommand{\e}{\mathbb{E}}
\newcommand{\p}{\mathbb{P}}

\newcommand{\vsi}{\boldsymbol{\sigma}}

\newtheorem{lemma}{\bf Lemma}

\newtheorem{theorem}{\bf Theorem}
\newtheorem{corollary}{\bf Corollary}
\newtheorem{remark}{\bf Remark}

\newtheorem{proposition}{\bf Proposition}

\newenvironment{Proof of lemma}{\noindent{\bf Proof of Lemma}}{\hfill$\Box$\newline}
\newenvironment{Proof of theorem}{\noindent{\bf Proof of Theorem}}{\hfill{\footnotesize${\square}$}\newline}
\newenvironment{Proof of theorems}{\noindent{\bf Proof of Theorems}}{\hfill$\Box$\newline}
\newenvironment{Proof of proposition}{\noindent{\bf Proof of Proposition}}{\hfill$\Box$\newline}
\newenvironment{Proof of propositions}{\noindent{\bf Proof of Propositions}}{\hfill$\Box$\newline}
\newenvironment{Proof of exercise}{\noindent{\it Proof of Exercise:}}{\hfill$\Box$}

\begin{document}

\title{The Parisi formula has a unique minimizer }
\author{Antonio Auffinger  \thanks{auffing@math.uchicago.edu. The research of A. A. is supported by NSF grant DMS-1407554.} \\ \small{University of Chicago}\and Wei-Kuo Chen \thanks{wkchen@math.uchicago.edu} \\ \small{University of Chicago} }
\maketitle

\begin{abstract}
In 1979, G. Parisi \cite{Par79} predicted a variational formula for the thermodynamic limit of the free energy in the Sherrington-Kirkpatrick model and described the role played by its minimizer. This formula was verified in the seminal work of Talagrand \cite{Tal06} and later generalized to the mixed $p$-spin models by Panchenko \cite{P12}. In this paper, we prove that the minimizer in Parisi's formula is unique at any temperature and external field by establishing the strict convexity of the Parisi functional.
\end{abstract}

\section{Introduction and main results}

The Sherrington-Kirkpatrick (SK) model was introduced in \cite{SK75}. For any $N\geq 1,$ its Hamiltonian at (inverse) temperature $\beta>0$ and external field $h\in\mathbb{R}$ is given by
$$
\frac{\beta}{\sqrt{N}}\sum_{i,j=1}^Ng_{ij}\sigma_i\sigma_j+h\sum_{i=1}^N\sigma_i
$$
for $\vsi=(\sigma_1,\ldots,\sigma_N)\in\Sigma_N:=\{-1,+1\}^N$, where $g_{ij}$'s are independent standard Gaussian random variables. It is arguably the most well-known model of disordered mean field spin glasses. Over the past few decades, its study has generated hundreds of papers in both theoretical physics and mathematics communities. We refer readers to the book of M\'{e}zard-Parisi-Virasoro \cite{MPV} for physics' methodologies and predictions and the books of Talagrand \cite{Tal11} and Panchenko \cite{P13} for its recent rigorous treatments. 

\smallskip

This paper is concerned with a generalization of the SK model, the so-called mixed $p$-spin model, which corresponds to the Hamiltonian 
\begin{equation}\label{eq:bqs}
H_{N}(\vsi) = H_N'(\vsi)+h\sum_{i=1}^N\sigma_i
\end{equation}
for $\vsi\in\Sigma_N$, where $$
H_{N}'(\vsi)=\sum_{p=2}^{\infty} \beta_p  H_{N,p}(\vsi)
$$
is the linear combination of the pure $p$-spin Hamiltonian,
\begin{equation}
  \label{eq:Hamiltonianpspin}
  H_{N,p}(\boldsymbol\sigma) = \frac{1}{N^{(p-1)/2}}
  \sum_{i_1, \dots, i_p=1}^N
  g_{i_1, \dots, i_p} \sigma_{i_1}\dots \sigma_{i_p}.
\end{equation}
 Here $g_{i_1, \dots, i_p}$'s are independent standard Gaussian random variables for all $p\geq 2$ and all $(i_1,\ldots,i_p).$ The nonnegative real sequence $(\beta_p)_{p\geq 2}$ is called the temperature parameters and $h\in\mathbb{R}$ denotes the strength of the external field.  We assume that $\beta_p>0$ for at least one $p\geq 2$ and $(\beta_p)_{p\geq 2}$ decreases fast enough, for instance, $\sum_{p=2}^\infty 2^p\beta_p^2<\infty$. The SK model can be recovered by choosing $\beta_p = 0$ for all $p\geq3$. A direct computation gives 
$$
\e H_{N}'(\vsi^1)H_N'(\vsi^2)=N\xi(R_{1,2}),
$$
where $R_{1,2}:=N^{-1}\sum_{i=1}^N\sigma_i^1\sigma_i^2$ is the overlap between spin configurations $\vsi^1$ and $\vsi^2$ and
 \begin{align}\label{eq-1}
 \xi(s):=\sum_{p=2}^\infty\beta_p^2s^p,\,\,\forall s\in[-1,1].
 \end{align}
Define the Gibbs measure as $
G_N(\vsi)=Z_N^{-1}\exp(-H_N(\vsi))
$
for $\vsi\in\Sigma_N$,
where the normalizing factor $Z_N$ is known as the partition function. 

\smallskip

Let $\mathcal{M}$ be the collection of all probability measures on $[0,1]$ endowed with the metric $d(\mu,\mu'):=\int_0^1|\mu([0,s])-\mu'([0,s])|ds$. Denote by $\mathcal{M}_d$ the collection of all atomic measures from $\mathcal{M}.$ Let $\mu\in \mathcal{M}_d$ with jumps at $\{q_p\}_{p=1}^{k+1}$ for some $k\geq 0.$ Set $q_0=0,$ $q_{k+2}=1$ and $m_l=\mu([0,q_l])$ for $0\leq l\leq k+1.$ We set a real-valued function $\Phi_\mu$ on $[0,1]\times\mathbb{R}$ as follows. Starting with $\Phi_\mu(1,x)=\Phi_{\mu}(q_{k+2},x)=\log\cosh x,$ define for $(s,x)\in [q_{k+1},q_{k+2})\times\mathbb{R},$
\begin{align}
\begin{split}\label{eq:ParPDE:eq1}
\Phi_\mu(s,x)&=\frac{1}{m_{k+1}}\log \e \exp m_{k+1}\Phi_\mu (q_{k+2}, x+z\sqrt{\xi'(q_{k+2})-\xi'(s)})\\
&=\log\cosh x+\frac{1}{2}(\xi'(1)-\xi'(s))
\end{split}
\end{align}
and decreasingly for $0\leq l\leq k,$ 
\begin{align}
\label{eq:ParPDE:eq2}
\Phi_\mu(s,x)=\frac{1}{m_l}\log \e \exp m_l\Phi_\mu (q_{l+1}, x+z\sqrt{\xi'(q_{l+1})-\xi'(s)})
\end{align}
for $(s,x)\in [q_l,q_{l+1})\times\mathbb{R}$, where $z$ is a standard Gaussian random variable. Two important properties about $\Phi_{\mu}$ are the following. First, it satisfies the Parisi PDE on each $[q_{l},q_{l+1})\times\mathbb{R},$
\begin{align}
 \label{eq:ParPDE}
 \partial_s \Phi_\mu(s,x) = - \frac{\xi''(s)}{2}\left(\partial_{xx}\Phi_\mu(s,x)+\mu([0,s])\ (\partial_x \Phi_\mu(s,x))^2\right).
 \end{align}
Second, it is well-known (see Guerra \cite{Guerra03}) that the mapping $\mu\mapsto \Phi_\mu$ indeed defines a Lipschitz functional from $(\mathcal{M}_d,d)$ to $(C([0,1]\times \mathbb R),\|~\cdot~\|_\infty)$. This allows us to extend $\Phi_\mu$ continuously to all $\mu\in\mathcal{M}.$ Note that $\Phi_\mu$ does not necessarily satisfy \eqref{eq:ParPDE} for arbitrary $\mu.$ In particular, if $\mu$ is continuous on $[0,1],$ then $\Phi_{\mu}$ is the classical solution to the Parisi PDE \eqref{eq:ParPDE} with terminal condition $\Phi_{\mu}(1,x)=\log\cosh x.$ Throughout this paper, we shall call, with a slightly abuse of notation, the function $\Phi_{\mu}$ for any $\mu\in\mathcal{M}$ the Parisi PDE solution. We now define Parisi's functional as
\begin{equation}\label{eq:ParFunctionalFormula}
\mathcal{P}_{\xi,h}(\mu)=\Phi_\mu(0,h)
 \end{equation}
 for $\mu\in\mathcal{M}.$ With these notations, the thermodynamic limit of the free energy can now be computed through

\begin{theorem}[Parisi formula \cite{Tal06, P12}]
\label{pf} We have almost surely,
\begin{equation}\label{ParisiFormulaA}
\lim_{N\rightarrow \infty} \frac{1}{N} \log Z_N = \min_{\mu\in \mathcal{M}}\left(\log 2+\mathcal{P}_{\xi,h}(\mu)-\frac{1}{2}\int_0^1s\xi''(s)\mu([0,s])ds\right).
\end{equation}
\end{theorem}

This formula was predicted in the ground-breaking work of Parisi \cite{Par79,Par80} in the setting of the SK model. It was proved and generalized by Talagrand \cite{Tal06} to the mixed even $p$-spin models after the celebrated discovery of the replica symmetric breaking bound by Guerra \cite{Guerra03}. Later Panchenko \cite{P12} verified its validity in the mixed $p$-spin models including odd $p.$ We shall call a minimizer of \eqref{ParisiFormulaA} a Parisi measure throughout this paper. Parisi's prediction goes beyond the variational formula. In his picture, the Parisi measure is unique. It  also describes the limiting distribution of the overlap $R_{1,2}$ under $\e G_N^{\otimes 2}$ and encodes all information of the model. Mathematically, uniqueness of the Parisi measure was only known in the generic case, that is, when $\beta_p >0$ for all $p\geq 2$ (see Theorem 1.2 (c) in Talagrand \cite{Tal07}) and in the spherical version of the present model \cite{Tal06s}. 
\smallskip

As can be seen from \eqref{ParisiFormulaA}, the third term on the right-hand side is linear in $\mu$. Therefore the proof of the uniqueness of the Parisi measure is related to Talagrand's conjecture \cite{Tal07, Tal06} that the functional $\mathcal P_{\xi,h}$ is strictly convex. The first partial result along this direction was presented in Panchenko \cite{P08} where he established convexity between measures that stochastically dominate each other; his result was later pushed forward slightly by Chen \cite{C13} using a PDE approach. In this paper, based on a variational representation for the Parisi PDE solution, we give a complete solution to Talagrand's conjecture:

 \begin{theorem} \label{thm2}
 For any $\xi$ and $h,$ the Parisi functional $\mathcal{P}_{\xi,h}$ is strictly convex.

 \end{theorem}

This directly implies

\begin{corollary}\label{cor1} For any $\xi$ and $h$, there exists a unique Parisi measure.
\end{corollary}

As an immediate consequence of Corollary \ref{cor1}, we remark that one can identify the high temperature regime of the model as the collection of all $(\beta_{p})_{p\geq 2}$ and $h\in\mathbb{R}$ such that the corresponding Parisi measure is a Dirac measure. We refer the reader to Theorem 13.4.1 in Talagrand \cite{Tal11} for the characterization of the high temperature regime in the SK model. For physicists' predictions and rigorous qualitative properties about the Parisi measure, one may consult authors' recent work \cite{AChen}. We also remark that the proof of Theorem \ref{thm2} and Corollary \ref{cor1} can be extended to other models that share a similar characterization of the limiting free energy including, for instance, the Ghatak-Sherrington model \cite{Ghan,PK}. We do not pursue this direction here.

\smallskip

This paper is motivated by a beautiful variational representation for certain functionals of Brownian motion due to Bou\'{e}-Dupuis \cite{BD98} and some techniques from the functional geometry introduced by Borell \cite{Borell}. In the simplest case, the Bou\'{e}-Dupuis representation states that if $f$ is a measurable function on $\mathbb{R}$ that is bounded from below, then
\begin{align}\label{eq--3}
\log\e\exp f(z)=\sup_{v}\e\biggl[f\biggl(B(1)+\int_0^1 v(s)ds\biggr)-\frac{1}{2}\int_0^1v(s)^2ds\biggr],
\end{align}
where $z$ is a standard Gaussian random variable, $(B(s))_{0\leq s\leq 1}$ is the standard Brownian motion and the supremum is taken over all progressively measurable processes $v$ with respect to the filtration generated by the Brownian motion $(B(s))_{0\leq s\leq 1}.$ To see how \eqref{eq--3} comes into play in the proof of the convexity of the Parisi functional, in view of \eqref{eq:ParPDE:eq1} and \eqref{eq:ParPDE:eq2}, it seems reasonable to ask whether the following function is convex or not,
$$
\phi(m)=\frac{1}{m}\log \e\exp mf(z)
$$
for $m>0,$ where $f$ is a convex function that is bounded from below. To answer this question, applying \eqref{eq--3} gives
\begin{align}
\begin{split}
\frac{1}{m}\log\e\exp mf(z)&=\sup_{v}\e\biggl[f\biggl(B(1)+\int_0^1 v(s)ds\biggr)-\frac{1}{2m}\int_0^1v(s)^2ds\biggr]\notag
\end{split}\\
\begin{split}\label{eq--4}
&=\sup_{u}\e\biggl[f\biggl(B(1)+m\int_0^1 u(s)ds\biggr)-\frac{m}{2}\int_0^1u(s)^2ds\biggr],
\end{split}
\end{align}
where the second equality used a change of variable $v=mu.$ Let $m_0,m_1>0,$ $\lambda\in[0,1]$ and take $m=(1-\lambda)m_0+\lambda m_1.$ Observe that from the first term inside the expectation of \eqref{eq--4},
\begin{align*}
f\biggl(B(1)+m\int_0^1 u(s)ds\biggr)&\leq (1-\lambda)f\biggl(B(1)+m_0\int_0^1 u(s)ds\biggr)\\
&\qquad+\lambda f\biggl(B(1)+m_1\int_0^1 u(s)ds\biggr)
\end{align*}
since $f$ is convex, while the second term satisfies
\begin{align}\label{prop:eq2}
m\int_0^1u(s)^2ds&=(1-\lambda)m_0\int_0^1u(s)^2ds+\lambda m_1\int_0^1u(s)^2ds.
\end{align}
Using \eqref{eq--4}, these imply the convexity of $\phi,$ $\phi(m)\leq (1-\lambda)\phi(m_0)+\lambda \phi(m_1).$ It is of independent interest that indeed exactly the same argument also yields the following Gaussian inequality.

\begin{proposition}
Suppose that $F,G,H$ are nonnegative measurable functions on $\mathbb{R}$. If 
\begin{align}\label{prop:eq1}
F((1-\lambda)x+\lambda y)\leq G(x)^{1-\lambda}H(y)^\lambda
\end{align}
for any $x,y\in\mathbb{R}$ and $\lambda\in[0,1],$   
then for any $m_0,m_1>0$ and $\lambda\in [0,1]$, we have
\begin{align*}
(\e F(z)^{m_\lambda})^{1/m_\lambda}&\leq (\e G(z)^{m_0})^{(1-\lambda)/m_0}(\e H(z)^{m_1})^{\lambda/m_1},
\end{align*}
where $m_\lambda:=(1-\lambda)m_0+\lambda m_1$ and $z$ is a standard Gaussian random variable.
\end{proposition}

\begin{proof}
Let $f=\log F,g=\log G$ and $h=\log H.$ Without loss of generality, we may assume that they are all bounded from below. From \eqref{prop:eq1}, 
\begin{align*}
f\biggl(B(1)+m_\lambda\int_0^1 u(s)ds\biggr)&\leq (1-\lambda)g\biggl(B(1)+m_0\int_0^1 u(s)ds\biggr)\\
&\qquad+\lambda h\biggl(B(1)+m_1\int_0^1 u(s)ds\biggr)
\end{align*}
This and \eqref{prop:eq2} completes our proof by using the variational formulas \eqref{eq--4} for $f,g,h$.
\end{proof}

The simplest case we explained above is essentially for proving the convexity of the Parisi functional for measures that are of the form $\mu=m\delta_0+(1-m)\delta_1$ for $0\leq m\leq 1,$ where $\delta_q$ denotes the Dirac measure at $q.$ We will show that using an iteration scheme together with an approximation procedure, there is a natural generalization of \eqref{eq--4} to the Parisi PDE solution for arbitrary measures. Thus, the argument above will allow us to deduce the convexity of the Parisi functional. The major difficulty throughout this paper is to justify the strict convexity. This part of the argument will need a more precise description of the optimizer in the variational problem and its subtle properties.
\smallskip

The rest of the paper is organized as follows. In Section \ref{sec2}, we establish a variational representation for the Parisi PDE solution. Moreover, we give an expression for the optimizer as well as a criterion for its uniqueness. Using these results, we will establish a general strict convexity for the Parisi PDE solution in Section \ref{sec3} and conclude immediately Theorem \ref{thm2}.
  
  \vspace{0.62cm}
\noindent {\bf Acknowledgements.} The second author thanks Dmitry Panchenko for introducing the problem of the strictly convexity of the Parisi functional to him. They are indebted to Louis-Pierre Arguin for fruitful discussions during the early stage of the project and to CIRM at Luminy for the hospitality during the workshop on disordered systems in 2013. Special thanks are due to Gregory Lawler for explaining the representation formula in Bou\'{e}-Dupuis \cite{BD98} to them. The authors also thank the anonymous referees for giving several suggestions and bringing \cite{F78} to their attention.


\section{A variational representation}\label{sec2}

Recall $\xi$ from \eqref{eq-1}. Define $\zeta=\xi''.$ Let $(B(r))_{r\geq 0}$ be a standard Brownian motion and $\p$ denote the Wiener measure. For $0\leq s \leq t \leq 1$, we denote by $\mathcal{D}[s,t]$ the space of all progressively measurable processes $u$ on $[s,t]$, with respect to the filtration generated by $(B(r))_{r\geq 0}$, that satisfy $\sup_{s\leq r\leq t}|u(r)|\leq 1$. We endow $\mathcal{D}[s,t]$ with the norm
$$\|u\|=\left(\e \int_s^t u(r)^2dr\right)^{1/2}.$$ 
The main result of this section is the following characterization.

\begin{theorem}[Variational formula] \label{thm0}
Let $\mu\in\mathcal{M}$ and $\alpha$ be its distribution function. Suppose that $\Phi$ is the Parisi PDE solution corresponding to $\mu$. Let $0\leq s<t\leq 1$. For any $x\in\mathbb{R}$ and $u\in \mathcal{D}[s,t],$ define
\begin{align}
\label{thm0:eq1}
F^{s,t}(u,x)=\e\left[C^{s,t}(u,x)-L^{s,t}(u)\right],
\end{align}
where
\begin{align}
\begin{split}
\label{thm0:eq2}
C^{s,t}(u,x)&=\Phi\left(t,x+\int_s^t\alpha(r)\zeta(r)u(r)dr+\int_s^t\zeta(r)^{1/2}dB(r)\right),\\
L^{s,t}(u)&=\frac{1}{2}\int_s^t\alpha(r)\zeta(r)u(r)^2dr.
\end{split}
\end{align} 
Then we have that
\begin{itemize} 
\item[$(i)$] \begin{equation}\label{MaxFormula}
\Phi(s,x)=\max\left\{F^{s,t}(u,x)|u\in \mathcal{D}[s,t]\right\}.
\end{equation}

\item[$(ii)$] The maximum in \eqref{MaxFormula} is attained by 
\begin{align}\label{max}
u^*(r)&=\partial_x\Phi(r,X(r)),
\end{align}
where $(X(r))_{s\leq r\leq t}$ is the strong solution to
\begin{align}
\begin{split}
\label{sde}
dX(r)&=\alpha(r)\zeta(r)\partial_x\Phi(r,X(r))dr+\zeta(r)^{1/2}dB(r),\\
X(s)&=x.
\end{split}
\end{align}
\end{itemize}

\end{theorem}

 \begin{remark}
\label{rmk1} \rm
$\partial_x\Phi$ and $\partial_{xx}\Phi$ exist and are continuous, see \eqref{prop2:eq1} below. As $|\partial_{xx}\Phi|\leq 1$ (see \eqref{prop2:eq3}), one obtains that for any $s,y_1,y_2,$
$$
|\alpha(s)\zeta(s)\partial_x\Phi(s,y_1)-\alpha(s)\zeta(s)\partial_x\Phi(s,y_2)|\leq \zeta(1)|y_1-y_2|.
$$ 
Using \cite[Proposition 2.13]{KS}, this ensures the existence of the strong solution $(X(r))_{s\leq r\leq t}$ for any $\mu$ and thus $u^*$ is a well-defined continuous process. Indeed, as one shall see from Lemma \ref{lem} below, $u^*$ is a continuous martingales. So one may as well replace $\mathcal{D}[s,t]$ by the space of all continuous martingale $u$ with respect to the filtration generated by the Brownian motion $(B(r))_{0\leq r\leq 1}$ and $\sup_{s\leq r\leq t}|u(r)|\leq 1.$
\end{remark}

Before we turn to the proof of Theorem \ref{thm0}, we summarize some properties about the Parisi PDE in the following proposition.

\begin{proposition}
\label{prop2}
Let $\mu\in\mathcal{M}$. Denote by $\alpha$ the distribution function and by $\Phi$ the Parisi PDE solution associated to $\mu$.  Then
\begin{itemize}
\item[$(i)$] For $0\leq j\leq 4,$ 
\begin{align}\label{prop2:eq1}
\mbox{$\partial_{x}^j\Phi$ exists and is continuous.}
\end{align} 
\item[$(ii)$] For all $(s,x)\in[0,1]\times\mathbb{R},$ 
\begin{align}
\begin{split}\label{prop2:eq2}
&|\partial_x\Phi(s,x)|\leq 1,
\end{split}\\
\begin{split}\label{prop2:eq3}
&\frac{C}{\cosh^2x}\leq \partial_{xx}\Phi(s,x)\leq 1,
\end{split}\\
\begin{split}\label{prop2:eq6}
&|\partial_{x}^3\Phi(s,x)|\leq 4,
\end{split}
\end{align}
where $C>0$ is a constant depending only on $\xi.$

\item[$(iii)$] If $\alpha$ is continuous on $[0,1],$ then 
\begin{align}
\label{prop2:eq4}
\Phi,\partial_x\Phi,\partial_{xx}\Phi\in \mathcal C^{1,2},
\end{align} 
where $\mathcal C^{1,2}$ is the space of all functions $f$ on $[0,1]\times\mathbb{R}$ with continuous $\partial_sf$ and $\partial_{xx}f.$ 

\item[$(iv)$] Suppose that $(\mu_n)_{n\geq 1}\in\mathcal{M}$ converges weakly to $\mu$ and $\Phi_n$ is the Parisi PDE solution associated to $\mu_n$. For $0\leq j\leq 2,$ uniformly on $[0,1]\times\mathbb{R}$,
\begin{align}\label{prop2:eq5}
\lim_{n\rightarrow\infty}\partial_{x}^j\Phi_n=\partial_{x}^j\Phi.
\end{align} 
\end{itemize}
\end{proposition}

\begin{proof}
Statements \eqref{prop2:eq1}, \eqref{prop2:eq2}, \eqref{prop2:eq3} and \eqref{prop2:eq5} are parts of the results of Proposition $1$ and $2$ in \cite{AChen}. As for \eqref{prop2:eq6}, it follows from $(14.272)$ in \cite{Tal11} and \eqref{prop2:eq5}. For \eqref{prop2:eq4}, note that the continuity of $\alpha$ gives that
\begin{align*}
\partial_s\Phi(s,x)&=-\frac{\zeta(s)}{2}(\partial_{xx}\Phi(s,x)+\alpha(s)(\partial_x\Phi(s,x))^2)
\end{align*}
is continuous for all $(s,x)\in[0,1]\times \mathbb{R}.$ This and \eqref{prop2:eq1} give \eqref{prop2:eq4}.
\end{proof}

The first step to prove Theorem \ref{thm0} is the following lemma about  atomic measures.

\begin{lemma}\label{lem2}
 Let $\mu\in\mathcal{M}_d$ and $\alpha$ be its distribution function. Let $s,t\in[0,1]$ with $s<t$ be both jump points of $\mu$. Using the notations of Theorem \ref{thm0}, we have 
\begin{align}
\label{lem1:eq1}
\Phi(s,x)&\geq \sup\{F^{s,t}(u,x)|u\in\mathcal{D}[s,t]\}.
\end{align}
\end{lemma}

\begin{proof} 
Suppose that $\mu$ has exactly $k+1$ jumps at $(q_{l})_{1\leq l\leq k+1}$ with $\mu([0,q_l])=m_l$, where $(q_l)_{1\leq l \leq k+1}$ and $(m_l)_{1\leq l \leq k+1}$ satisfy
 \begin{align*}
 &0\leq m_1<m_2<\cdots<m_{k}< m_{k+1}=1,\\
 &0\leq q_1<q_2<\cdots<q_k<q_{k+1}\leq 1.
 \end{align*}
 and that for some $1\leq a< b \leq k+1$ we have $$q_a= s,\,\,q_{b} = t.$$
Set $m_0 = q_0 =0$ and $q_{k+2}=1$. As we have discussed in the Section $1$, $\Phi$ is defined through \eqref{eq:ParPDE:eq1} and \eqref{eq:ParPDE:eq2}. Using the standard Brownian motion $(B(r))_{r\geq 0}$, for each $a\leq l\leq b-1,$ we can write
\begin{align*}
\Phi(q_l,x)&=\frac{1}{m_l}\log \e\exp m_l\Phi\left(q_{l+1},x+\int_{q_l}^{q_{l+1}}\zeta(r)^{1/2}dB(r)\right).
\end{align*}
Let $u\in\mathcal{D}[s,t].$ Set
\begin{align*}
Z_l&=\exp\left(-\frac{1}{2}\int_{q_l}^{q_{l+1}}m_l^2u(r)^2\zeta(r)dr-\int_{q_l}^{q_{l+1}}m_lu(r)\zeta(r)^{1/2}dB(r)\right).
\end{align*}
Define $d\tilde{\p}=Z_ld\p$ and $\tilde{B}(r)=\int_{q_l}^{r}m_lu(a)\zeta(a)^{1/2}da+B(r).$ We use $\tilde{\e}$ to denote the expectation with respect to $\tilde{\p}.$
The Girsanov theorem \cite[Theorem 5.1]{KS} says 
\begin{align*}
&\e\exp m_l\Phi\left(q_{l+1},x+\int_{q_l}^{q_{l+1}}\zeta(r)^{1/2}dB(r)\right)\\
&=\tilde{\e}\exp m_l\Phi\left(q_{l+1},x+\int_{q_l}^{q_{l+1}}\zeta(r)^{1/2}d\tilde{B}(r)\right)\\
&=\e\left[\exp m_l\Phi\left(q_{l+1},x+\int_{q_l}^{q_{l+1}}m_lu(r)\zeta(r)dr+\int_{q_l}^{q_{l+1}}\zeta(r)^{1/2}d{B}(r)\right)\right.\\
&\qquad \left.\cdot \exp\left(-\frac{1}{2}\int_{q_l}^{q_{l+1}}m_l^2u(r)^2\zeta(r)dr-\int_{q_l}^{q_{l+1}}m_lu(r)\zeta(r)^{1/2}dB(r)\right)\right].
\end{align*}
From Jensen's inequality $m^{-1}\log \e \exp mA\geq \e A$ for any random variable $A$ and  $m>0$ and  noting that $m_l=\alpha(r)$ for $q_l\leq r<q_{l+1}$, it follows
\begin{align*}
\Phi(q_l,x)
&\geq\e \left[\Phi\left(q_{l+1},x+\int_{q_l}^{q_{l+1}}\alpha(r)u(r)\zeta(r)dr+\int_{q_l}^{q_{l+1}}\zeta(r)^{1/2}d{B}(r)\right)\right.\\
&\qquad\qquad\qquad\left.
-\frac{1}{2}\int_{q_l}^{q_{l+1}}\alpha(r)u(r)^2\zeta(r)dr\right]
\end{align*}
for all $a\leq l\leq b-1.$ Using this and conditional expectation, an iteration argument on $l$ from $a$ to $b-1$ gives
\begin{align*}
\Phi(s,x)&=\Phi(q_a,x)\\
&\geq \e\left[\Phi\left(q_{b},x+\sum_{l=a}^{b-1}\int_{q_l}^{q_{l+1}}\alpha(r)u(r)\zeta(r)dr+\sum_{l=a}^{b-1}\int_{q_l}^{q_{l+1}}\zeta(r)^{1/2}d{B}(r)\right)\right.\\
&\qquad\qquad\qquad\left.
-\frac{1}{2}\sum_{l=a}^{b-1}\int_{q_l}^{q_{l+1}}\alpha(r)u(r)^2\zeta(r)dr\right]\\
&=F^{q_a,q_{b}}(u,x)\\
&=F^{s,t}(u,x).
\end{align*}
Since this is true for arbitrary $u\in\mathcal{D}[s,t],$ this finishes our proof.
\end{proof}

\begin{proof}[Proof of Theorem \ref{thm0}]
First, we claim that for any $\mu\in\mathcal{M},$
\begin{align}
\label{eq-4}
\Phi(s,x)\geq \sup\left\{F^{s,t}(u,x)|u\in\mathcal{D}[s,t]\right\}.
\end{align}
Pick a sequence $(\mu_n)_{n\geq 1}\in \mathcal{M}_d$ that converges weakly to $\mu$ and have jumps at $s$ and $t$. Denote by $\alpha_n$ the distribution function and by $\Phi_n$ the Parisi PDE solution associated to $\mu_n$. Since $(\alpha_n)_{n\geq 1}$ converges almost everywhere to $\alpha,$ the uniform boundedness of $u\in\mathcal{D}[s,t]$ and the dominated convergence theorem give
\begin{align*}
\int_s^t \alpha_n(r) \zeta(r)u(r) dr \rightarrow \int_s^t \alpha(r) \zeta(r)u(r) dr,\\
 \int_{s}^{t} \alpha_n(r) \zeta(r)u(r)^2 dr  \rightarrow   \int_{s}^{t} \alpha(r) \zeta(r)u(r)^2 dr
\end{align*}
almost surely.
From \eqref{prop2:eq5}, it implies that the sequence of functionals $(F_n^{s,t})_{n\geq 1}$ associated to $(\alpha_n)_{n\geq 1}$ converges uniformly to $F^{s,t}$ and therefore Lemma \ref{lem2} gives \eqref{eq-4}. 

\smallskip

With the help of this claim, our proof will be finished if we show that $u^*$ is a maximizer and equality of \eqref{eq-4} holds. We check the case that $\alpha$ is continuous first. 
Define for $s\leq r\leq t,$
\begin{align*}
Y(r)&=\Phi(r,X(r))-\frac{1}{2}\int_s^r\alpha(v)\zeta(v)u^*(v)^2dv-\int_{s}^ru^*(v)\zeta(v)^{1/2}dB(v).
\end{align*}
Since $\Phi\in \mathcal C^{1,2}$ by \eqref{prop2:eq4}, we obtain from It\^{o}'s formula \cite[Theorem 3.6]{KS} and \eqref{eq:ParPDE},
\begin{align*}
d\Phi&=(\partial_s\Phi)dr+(\partial_x\Phi) dX+\frac{1}{2}(\partial_{xx}\Phi) \zeta dr\\
&=\left(\partial_s\Phi +\alpha\zeta\left(\partial_x\Phi \right)^2+\frac{1}{2}\zeta (\partial_{xx}\Phi) \right)dr+\zeta^{1/2}(\partial_x\Phi)dB\\
&=\frac{1}{2}\alpha\zeta\left(\partial_x\Phi\right)^2 dr+\zeta^{1/2}(\partial_x\Phi) dB\\
&=\frac{1}{2}\alpha\zeta (u^*)^2dr+\zeta^{1/2}u^*dB,
\end{align*}
which implies
\begin{align}\label{eq6}
dY&=d\Phi-\frac{1}{2}\alpha \zeta(u^*)^2dr-\zeta^{1/2}u^*dB=0.
\end{align}
Since
\begin{align*}
Y(s)&=\Phi(s,X(s))=\Phi(s,x)
\end{align*}
and
\begin{align*}
Y(t)=\Phi(t,X(t))-\frac{1}{2}\int_s^t\alpha(v)\zeta(v)u^*(v)^2dv-\int_{s}^tu^*(v)\zeta(v)^{1/2}dB(v),
\end{align*}
if we take expectation in the equation above, it follows that from \eqref{eq6}, 
\begin{align}\label{eq-7}
F^{s,t}(u^*,x)=\e Y(t)=\e Y(s)+\e \int_s^tdY(r)=\Phi(s,x).
\end{align}
This means that $u^*$ is a maximizer. As for arbitrary $\alpha,$ let us pick a sequence of probability measures $(\mu_n)_{n\geq 1}\subset\mathcal{M}$ such that the distribution function of each $\mu_n$ is continuous and $(\mu_n)_{n\geq 1}$ has weak limit $\mu$. Denote by $(\alpha_n)_{n\geq 1}$, $(\Phi_n)_{n\geq 1}$, $(X_n)_{n\geq 1}$ and $(u_n^*)_{n\geq 1}$ the distribution functions, the Parisi PDE solutions, the SDE solutions \eqref{sde} and the maximizers \eqref{max} associated to $(\alpha_n)_{n\geq 1}.$ Write 
\begin{equation*}
\begin{split}
&|X_n(r) - X(r)|\\
 &= \bigg |\int_{s}^r \alpha_n(v)\zeta(v) \partial_x \Phi_n(v,X_n(v)) dv - \int_{s}^r \alpha(v) \zeta(v)\partial_x \Phi(v,X(v)) dv\bigg| \\
&\leq \int_{s}^r \zeta(v)|\alpha_n(v) - \alpha(v)| |\partial_x \Phi_n(v,X(v))| + \alpha(v)\zeta(v)|\partial_x \Phi_n(v,X_n(v)) - \partial_x \Phi(v,X(v))| dv \\
& \leq C \int_{s}^r |\alpha_n(v) - \alpha(v)| |\partial_x \Phi_n(v,X(v))| + |\partial_x \Phi_n(v,X_n(v)) - \partial_x \Phi(v,X(v))| dv
\end{split}
\end{equation*}
for some $C>0$, where in the last line we used the fact that $\zeta$ is bounded above by $\zeta(1)$ and $\alpha \leq 1$.
From \eqref{prop2:eq2} and the almost everywhere convergence of $(\alpha_n)_{n\geq 1}$, the dominated convergence theorem tells us that the first term in the last line converges to zero. As for the second term, the mean value theorem, \eqref{prop2:eq3} and \eqref{prop2:eq5} imply that for given $\varepsilon >0$, if $n$ is large enough,
\begin{equation*}
\begin{split}
& \int_s^r\alpha(v)|\partial_x \Phi_n(v,X_n(v)) - \partial_x \Phi(v,X(v))| dv\\
&\leq   \int_{s}^r|\partial_x \Phi(v,X_n(v)) - \partial_x \Phi_n(v,X_n(v))| dv +\int_{s}^r|\partial_x \Phi(v,X(v)) - \partial_x \Phi(v,X_n(v))| dv \\
&\leq  \varepsilon+ \int_{s}^r |X_n(v)-X(v)| dv 
\end{split}
\end{equation*}
for all $s\leq r\leq t.$ Applying the Gronwall inequality, we conclude that for large enough $n$, $|X_n(r)-X(r)| \leq \varepsilon e^{Cr}$ for all $s\leq r\leq t$. Thus, $$
\lim_{n\rightarrow\infty}\sup_{s\leq r\leq t}|X_n(r)-X(r)|=0.
$$
and so from \eqref{prop2:eq5}, $(u_n^*)_{n\geq 1}$ converges to $u^*$ uniformly on $[s,t]$. This combining with \eqref{prop2:eq5} and the almost everywhere convergence of $(\alpha_n)_{n\geq 1}$ to $\alpha$ implies that \eqref{eq-7} is also true for $\mu.$ In other words, $u^*$ is a maximizer. This ends our proof.
\end{proof}

\begin{remark}
\rm  After submission of the paper, it came to our attention that the arguments of Lemma \ref{lem2} and Theorem \ref{thm0} seem to be very similar to the proof of the finite time analogue of \cite[Theorem 2.1]{F78} on pages $341-342.$
\end{remark}

\begin{proposition}[Uniqueness of $u^*$]
\label{prop1}
Let $\mu\in\mathcal{M}$ and $\alpha$ be its distribution function. Let $0<s<t\leq 1$ with $\alpha(s)>0.$ Suppose that $\Phi$ is the Parisi PDE solution corresponding to $\mu.$ If $\int_s^t\alpha(r)\zeta(r)dr<1$, then the maximizer $u^*$ given by \eqref{max} for the variational representation \eqref{MaxFormula} is unique.
\end{proposition}

\begin{proof}
Suppose that $\int_s^t\alpha(r)\zeta(r)dr<1$. It suffices to prove that $F^{s,t}(\cdot, x)$ defines a strictly concave functional on $\mathcal{D}[s,t].$
Let $u_0,u_1\in \mathcal{D}[s,t]$ with $u_0\neq u_1.$ This implies that
\begin{align}\label{eq17}
\|u_0-u_1\|^2&=\e\int_s^t|u_0(r)-u_1(r)|dr>0.
\end{align}
Note that $u_0$ and $u_1$ are uniformly bounded above by one. Define $u_\lambda=(1-\lambda) u_0+\lambda u_1$ for $\lambda\in[0,1].$ A direct computation using the dominated convergence theorem gives 
\begin{align*}
\partial_{\lambda\lambda}F^{s,t}(u_\lambda,x)&=
\e\left[\partial_{xx}\Phi\left(t,Z\right)\left(\int_s^t\alpha(r)\zeta(r)(u_1(r)-u_0(r))dr\right)^2\right]\\
&\qquad-\e \left[\int_s^t\alpha(r)\zeta(r)(u_1(r)-u_0(r))^2dr\right],
\end{align*}
where $Z=x+\int_s^t\alpha(r)\zeta(r)u_\lambda(r)dr+\int_s^t\zeta(r)^{1/2}dB(r).$
Note that from the Cauchy-Schwarz inequality,
\begin{align*}
\left(\int_s^t\alpha(r)\zeta(r)(u_1(r)-u_0(r))dr\right)^2&\leq \int_s^t\alpha(r)\zeta(r)dr\cdot\int_s^t\alpha(r)\zeta(r)(u_1(r)-u_0(r))^2dr.
\end{align*}
Using this and \eqref{prop2:eq3}, it follows that
\begin{align*}
&\partial_{\lambda\lambda}F^{s,t}(u_\lambda,x)\\
&\leq \e\left[\left(\int_s^t\alpha(r)\zeta(r)(u_1(r)-u_0(r))dr\right)^2-\int_s^t\alpha(r)\zeta(r)(u_1(r)-u_0(r))^2dr\right]\\
&\leq \e\left[\int_s^t\alpha(r)\zeta(r)dr \cdot \int_s^t\alpha(r)\zeta(r)(u_1(r)-u_0(r))^2dr-\int_s^t\alpha(r)\zeta(r)(u_1(r)-u_0(r))^2dr\right]\\
&=\left(\int_s^t\alpha(r)\zeta(r)dr-1\right)\e\left[\int_s^t\alpha(r)\zeta(r)(u_1(r)-u_0(r))^2dr\right].
\end{align*}
Note that since $\alpha$ is nondecreasing and $\alpha(s)>0$, we have $\alpha(r)>0$ for all $s\leq r\leq t.$ Also note that $\zeta(r)>0$ for all $0<r<t.$ Consequently, using \eqref{eq17} and $\int_s^t\alpha(r)\zeta(r)dr<1$, we conclude that $\partial_{\lambda\lambda}F^{s,t}(u_\lambda,x)<0$ and this gives the strict concavity of $F^{s,t}(\cdot,x).$ 
\end{proof}

\section{Strict convexity of the Parisi PDE solution}\label{sec3}

Suppose that $\mu_0,\mu_1\in\mathcal{M}$ and $x_0,x_1\in\mathbb{R}.$ For $\lambda\in[0,1],$ we set
\begin{align*}
\mu_\lambda&=(1-\lambda)\mu_0+\lambda\mu_1,\\
x_\lambda&=(1-\lambda)x_0+\lambda x_1.
\end{align*} 
Denote by $\alpha_0,\alpha_\lambda,\alpha_1$ the distribution functions and by $\Phi_0,\Phi_\lambda,\Phi_1$ the Parisi PDE solutions corresponding to $\mu_0,\mu_\lambda,\mu_1$, respectively. Let $\tau$ be the last time that $\alpha_0$ and $\alpha_1$ are different, that is,
\begin{align*}
\tau&=\min\left\{s\in[0,1]:\alpha_0(r)=\alpha_1(r),\,\,\forall r\in [s,1]\right\}.
\end{align*}
Note that since $\alpha_0(1)=\alpha_1(1)=1$ and $\alpha_0,\alpha_1$ are right continuous, $\tau$ is well-defined and that if $\mu_0\neq \mu_1$, then $\tau>0.$ The following general result immediately implies Theorem \ref{thm2} by letting $s=0$ and $x_0=x_1=h$ in $(ii)$ below.

\begin{theorem}\label{thm1} We have that
\begin{itemize}
\item[$(i)$] For any $\mu_0,\mu_1\in\mathcal{M},$ 
\begin{align}
\label{thm1:eq1}
\Phi_\lambda(s,x_\lambda)&\leq (1-\lambda)\Phi_0(s,x_0)+\lambda\Phi_1(s,x_1)
\end{align}
for all $\lambda,s\in[0,1]$ and $x_0,x_1\in\mathbb{R}.$

\item[$(ii)$] Suppose that $\mu_0,\mu_1$ are distinct. Then for any $0\leq s<\tau,$ the inequality \eqref{thm1:eq1} is strict for all $\lambda\in(0,1)$ and $x_0,x_1\in\mathbb{R}$.
\end{itemize}
\end{theorem}

%

As one shall see, while the statement \eqref{thm1:eq1} follows directly from our representation theorem, the proof for the strict inequality of \eqref{thm1:eq1} is more delicate and is based on subtle properties of the maximizers that we summarize as follows.

\begin{lemma}\label{lem}
Let $\mu\in\mathcal{M}$. Denote by $\alpha$ the distribution function and by $\Phi$ the Parisi PDE solution corresponding to $\mu$. Let $0\leq s\leq t\leq 1$ and $x\in\mathbb{R}$. Suppose that $(X(r))_{s\leq r\leq t}$ satisfies
\begin{align}
\begin{split}\label{eq0}
dX(r)&=\alpha(r)\zeta(r)\partial_x\Phi(r,X(r))dr+\zeta(r)^{1/2}dB(r),\\
X(s)&=x.
\end{split}
\end{align}
Then for any $s\leq a\leq b\leq t$, we have
\begin{align}\label{eq10}
\partial_x\Phi(b,X(b))-\partial_x\Phi(a,X(a))&=\int_a^b\zeta(r)^{1/2}\partial_{xx}\Phi(r,X(r))dB(r)
\end{align}
and
\begin{align}
\begin{split}\label{eq11}
&\partial_{xx}\Phi(b,X(b))-\partial_{xx}\Phi(a,X(a))\\
&=-\int_a^b\alpha(r)\zeta(r)(\partial_{xx}\Phi(r,X(r)))^2dr+\int_a^b\zeta(r)^{1/2}\partial_{x}^3\Phi(r,X(r))dB(r),
\end{split}
\end{align}
where the last It\^{o}'s integral is well-defined by \eqref{prop2:eq6}.
\end{lemma}

\begin{proof}
By an approximation argument as we did in the proof of Theorem \ref{thm0}, it suffices to assume that $\alpha$ is continuous on $[0,1].$ From \eqref{prop2:eq4}, this assumption ensures that $\Phi, \partial_x\Phi, \partial_{xx}\Phi~\in~\mathcal C^{1,2}$. Recall that $\Phi$ satisfies
\begin{align*}
\partial_s\Phi&=-\frac{\zeta}{2}(\partial_{xx}\Phi+\alpha(\partial_{x}\Phi)^2).
\end{align*}
A direct computation yields
\begin{align*}
\partial_{xs}\Phi&=-\frac{1}{2}\zeta(\partial_{x}^3\Phi)-\alpha\zeta (\partial_{xx}\Phi)(\partial_x\Phi)\\
\partial_{xxs}\Phi&=-\frac{1}{2}\zeta(\partial_{x}^4\Phi)-\alpha\zeta (\partial_{x}^3\Phi)(\partial_x\Phi)-\alpha\zeta(\partial_{xx}\Phi)^2.
\end{align*}
Now, using It\^{o}'s formula \cite[Theorem 3.6]{KS} and these two equations,
\begin{align*}
d(\partial_x\Phi)&=(\partial_{xs}\Phi) dr+(\partial_{xx}\Phi) dX+\frac{1}{2}\zeta(\partial_{x}^3\Phi) dr\\
&=\left(-\frac{1}{2}\zeta(\partial_{x}^3\Phi)-\alpha \zeta (\partial_{xx}\Phi)(\partial_x\Phi)\right)dr\\
&\,\,+(\partial_{xx}\Phi) \left(\alpha \zeta(\partial_x\Phi)dr+\zeta^{1/2}dB\right)+\frac{1}{2}\zeta(\partial_{x}^3\Phi )dr\\
&=\zeta^{1/2}(\partial_{xx}\Phi)dB
\end{align*}
and
\begin{align*}
d(\partial_{xx}\Phi)&=(\partial_{xxs}\Phi)dr+(\partial_{x}^3\Phi)dX+\frac{1}{2}\zeta(\partial_{x}^4\Phi)dr\\
&=\left(-\frac{1}{2}\zeta(\partial_{x}^4\Phi)-\alpha \zeta(\partial_{x}^3\Phi)(\partial_x\Phi)-\alpha \zeta(\partial_{xx}\Phi)^2\right)dr\\
&+(\partial_{x}^3\Phi) \left(\alpha\zeta(\partial_x\Phi)dr+\zeta^{1/2}dB\right)+\frac{1}{2}\zeta(\partial_{x}^4\Phi)dr\\
&=-\alpha\zeta(\partial_{xx}\Phi)^2dr+\zeta^{1/2}(\partial_{x}^3\Phi)dB.
\end{align*}
These two equations complete our proof of Lemma \ref{lem}.
\end{proof}

\begin{proof}[Proof of Theorem \ref{thm1}] Let $\mu_0,\mu_1\in\mathcal{M}$, $x_0,x_1\in\mathbb{R}$, $\lambda\in[0,1]$, $0\leq s\leq t\leq 1$ and $u\in\mathcal{D}[s,t].$ Recall Theorem \ref{thm0} and set $\mu_\lambda = (1-\lambda)\mu_0 + \lambda\mu_1$. For $\theta = 0, \lambda, 1$, denote by 
\begin{align*}
F_\theta^{s,t},C_\theta^{s,t},L_\theta^{s,t}, 
\end{align*}
the functionals defined in the variational formulas corresponding respectively to $\mu_\theta$:
\begin{align}
\begin{split}
\label{eq16}
\Phi_\theta(s,x_\theta)&=\max\{F_\theta^{s,t}(u,x_\theta)|\mathcal{D}[s,t]\}.
\end{split}
\end{align}

\smallskip

To show $(i)$, let $0\leq s\leq 1$ and take $t=1$. Suppose that $u\in \mathcal{D}[s,1].$ Note that
\begin{align}
\begin{split}\label{eq5}
L_\lambda^{s,1}(u)&=\int_s^1\alpha_\lambda(r)\zeta(r)u(r)^2dr=(1-\lambda) L_0^{s,1}(u)+\lambda L_1^{s,1}(u).
\end{split}
\end{align}
Write 
\begin{align}
\begin{split}
\label{eq19}
&x_\lambda+\int_s^1\alpha_\lambda(r)\zeta(r)u(r)dr+\int_s^1\zeta(r)^{1/2}dB(r)\\
&=(1-\lambda)\left(x_0+\int_s^1\alpha_0(r)\zeta(r)u(r)dr+\int_s^1\zeta(r)^{1/2}dB(r)\right)\\
&+\lambda\left(x_1+\int_s^1\alpha_1(r)\zeta(r)u(r)dr+\int_s^1\zeta(r)^{1/2}dB(r)\right).
\end{split}
\end{align}
Since $\Phi_0(1,x)=\Phi_\lambda(1,x)=\Phi_1(1,x)=\log\cosh x$ is a convex function, we obtain
\begin{align*}
C_\lambda^{s,1}(u,x_\lambda)&=\Phi_\lambda\left(1,x_\lambda+\int_s^1\alpha_\lambda(r)\zeta(r)u(r)dr+\int_{s}^1\zeta(r)^{1/2}dB(r)\right)\\
&\leq(1-\lambda)\Phi_\lambda\left(1,x_0+\int_s^1\alpha_0(r)\zeta(r)u(r)dr+\int_{s}^1\zeta(r)^{1/2}dB(r)\right)\\ 
&+\lambda\Phi_\lambda\left(1,x_1+\int_s^1\alpha_1(r)\zeta(r)u(r)dr+\int_{s}^1\zeta(r)^{1/2}dB(r)\right)\\
&=(1-\lambda) C_0^{s,1}(u,x_0)+\lambda C_1^{s,1}(u,x_1).
\end{align*}
Combining with \eqref{eq5} and taking expectation, one has 
\begin{align*}
F_\lambda^{s,1}(u,x_\lambda)&\leq (1-\lambda) F_0^{s,1}(u,x_0)+\lambda F_1^{s,1}(u,x_1).
\end{align*}
Since this is true for any $u\in \mathcal{D}[s,1],$ the representation formula \eqref{eq16} gives
$(i)$.

\smallskip

Next, we turn to the proof of $(ii).$ Suppose that $\alpha_0\neq \alpha_1$ and $\lambda\in(0,1)$. We will show that first, there exists some $\tau'\in(0,\tau)$ such that \eqref{thm1:eq1} is strict for all $s\in [\tau',\tau)$ and then, prove \eqref{thm1:eq1} is also strict for all $s\in [0,\tau').$ The way of finding such $\tau'$ can be argued as follows. Note that if $\int_{s}^\tau\alpha_0(r)dr=\int_{s}^\tau\alpha_1(r)dr=0$ for all $s\in[0,\tau),$ then $\alpha_0=\alpha_1=0$ on $[0,\tau)$ since $\alpha_0$ and $\alpha_1$ are nondecreasing. This implies that $\alpha_0=\alpha_1$ on $[0,1]$, which contradicts that $\alpha_0\neq \alpha_1.$ Therefore, there exists $\tau'< \tau$ such that at least one of  the integrals $\int_{\tau'}^{\tau} \alpha_0(r) dr, \int_{\tau'}^{\tau} \alpha_1(r) dr$ is nonzero. Without loss of generality this implies that $\alpha_0(\tau')>0$. Making $\tau'$ bigger if necessary we see that we can choose $0<\tau' < \tau$ such that the following statement holds:
\begin{align}
\label{proof:thm1:eq2}
\mbox{$\alpha_0(\tau')>0$ and $\int_{\tau'}^{\tau}\alpha_0(r) \zeta(r)dr<1$.}
\end{align}

Now we argue by contradiction. Suppose equality in \eqref{thm1:eq1} holds for some $s\in[\tau',\tau)$ and $x_0,x_1.$ Take $t=\tau$. Let $u_0^*,u_\lambda^*,u_1^*$ be the corresponding maximizers of \eqref{eq16} generated by \eqref{max}. As in \eqref{eq5},
\begin{align*}
L_\lambda^{s,t}(u_\lambda^*)&=(1-\lambda) L_0^{s,t}(u_\lambda^*)+\lambda L_1^{s,t}(u_\lambda^*).
\end{align*}
Also writing 
$$
x_\lambda+\int_s^t\alpha_\lambda(r)\zeta(r)u_\lambda^*(r)dr+\int_s^t\zeta(r)^{1/2}dB(r)
$$
in the same away as \eqref{eq19}, $(i)$ gives
\begin{align*}
C_\lambda^{s,t}(u_\lambda^*,x_\lambda)&\leq(1-\lambda)C_0^{s,t}(u_\lambda^*,x_0)+\lambda C_1^{s,t}(u_\lambda^*,x_1).
\end{align*}
They together imply that
\begin{align}
\begin{split}
\label{eq12}
F_\lambda^{s,t}(u_\lambda^*,x_\lambda)&\leq (1-\lambda) F_0^{s,t}(u_\lambda^*,x_0)+ \lambda F_1^{s,t}(u_\lambda^*,x_1).
\end{split}
\end{align}
Note that
\begin{align*}
F_0^{s,t}(u_\lambda^*,x_0)\leq \Phi_0(s,x_0),\,\,F_1^{s,t}(u_\lambda^*,x_1)\leq \Phi_1(s,x_1),\,\,
F_\lambda^{s,t}(u_\lambda^*,x_\lambda)=\Phi_\lambda(s,x_\lambda).
\end{align*}
Consequently, from \eqref{eq12} and the assumption
$$
\Phi_\lambda(s,x_\lambda)=(1-\lambda)\Phi_0(s,x_0)+\lambda\Phi_1(s,x_1),
$$
we obtain
\begin{align*}
F_0^{s,t}(u_\lambda^*,x_0)=\Phi_0(s,x_0),\\
F_1^{s,t}(u_\lambda^*,x_1)=\Phi_1(s,x_1).
\end{align*}
In other words, $u_\lambda^*$ realizes the maxima of the representations for $\Phi_0(s,x_0)$ and $\Phi_1(s,x_1)$. Now from \eqref{proof:thm1:eq2} and Proposition \ref{prop1}, we conclude uniqueness of the maximizer for $\Phi_0(s,x_0)$, that is, $u_0^*=u_\lambda^*$ with respect to the norm $\|\cdot\|.$ Since $u_0^*$ and $u_\lambda^*$ are continuous on $[s,t],$ we have
\begin{align}
\label{eq13}
\partial_x\Phi_0(r,X_0(r))=u_0^*(r)=u_\lambda^*(r)=\partial_x\Phi_\lambda(r,X_\lambda(r))
\end{align}
for all $s\leq r\leq t,$ where $(X_0(r))_{s\leq r\leq t}$ and $(X_\lambda(r))_{s\leq r\leq t}$ satisfy respectively,
\begin{align*}
dX_0(r)&=\alpha_0(r)\zeta(r)\partial_x\Phi_0(r,X_0(r))dr+\zeta(r)^{1/2}dB(r),\\
X_0(s)&=x_0,\\
dX_\lambda(r)&=\alpha_\lambda(r)\zeta(r)\partial_x\Phi_\lambda(r,X_\lambda(r))dr+\zeta(r)^{1/2}dB(r),\\
X_\lambda(s)&=x_\lambda.
\end{align*}
From \eqref{eq10} and \eqref{eq13},
\begin{align*}
\int_s^t\zeta(r)^{1/2}\partial_{xx}\Phi_0(r,X_0(r))dB(r)
&=\partial_x\Phi_0(t,X_0(t))-\partial_x\Phi_0(s,X_0(s))\\
&=\partial_x\Phi_\lambda(t,X_\lambda(t))-\partial_x\Phi_\lambda(s,X_\lambda(s))\\
&=\int_s^t\zeta(r)^{1/2}\partial_{xx}\Phi_\lambda(r,X_\lambda(r))dB(r).
\end{align*}
This gives by It\^{o}'s isometry,
\begin{align*}
&\int_s^t\zeta(r)\e\bigl(\partial_{xx}\Phi_0(r,X_0(r))-\partial_{xx}\Phi_\lambda(r,X_\lambda(r))\bigr)^2dr\\
&=\e\left(\int_s^t\zeta(r)^{1/2}\bigl(\partial_{xx}\Phi_0(r,X_0(r))-\partial_{xx}\Phi_\lambda(r,X_\lambda(r))\bigr)dB(r)\right)^2=0
\end{align*}
and therefore, by the continuity of $\partial_{xx}\Phi_0(\cdot,X_0(\cdot))$ and $\partial_{xx}\Phi_\lambda(\cdot,X_\lambda(\cdot))$ on $[s,t],$ we obtain
\begin{align}\label{eq14}
\partial_{xx}\Phi_0(r,X_0(r))=\partial_{xx}\Phi_\lambda(r,X_\lambda(r)),\,\,\forall r\in[s,t].
\end{align}
Next we use \eqref{eq11} and \eqref{eq14} to get that for all $a,b$ satisfying $s\leq a\leq b\leq t,$
\begin{align*}
&-\int_a^b\alpha_0(r)\zeta(r)\e(\partial_{xx}\Phi_0(r,X_0(r)))^2dr\\
&=\e\left(-\int_a^b\alpha_0(r)\zeta(r)(\partial_{xx}\Phi_0(r,X_0(r)))^2dr+\int_a^b\zeta(r)^{1/2}\partial_{x}^3\Phi_0(r,X_0(r))dB(r)\right)\\
&=\e\bigl(\partial_{xx}\Phi_0(b,X_0(b))-\partial_{xx}\Phi_0(a,X_0(a))\bigr)\\
&=\e\bigl(\partial_{xx}\Phi_\lambda(b,X_\lambda(b))-\partial_{xx}\Phi_\lambda(a,X_\lambda(a))\bigr)\\
&=\e\left(-\int_a^b\alpha_\lambda(r)\zeta(r)(\partial_{xx}\Phi_\lambda(r,X_\lambda(r)))^2dr+\int_a^b\zeta(r)^{1/2}\partial_{x}^3\Phi_\lambda(r,X_\lambda(r))dB(r)\right)\\
&=-\int_a^b\alpha_\lambda(r)\zeta(r)\e(\partial_{xx}\Phi_\lambda(r,X_\lambda(r)))^2dr.
\end{align*}
Note that by \eqref{prop2:eq3}, $\partial_{xx}\Phi_0(r,X_0(r))$ and $\partial_{xx}\Phi_\lambda(r,X_\lambda(r))$ are positive continuous functions on $[s,t]$ and that $\zeta>0$ on $[s,t].$ If $\alpha_0(a)\neq \alpha_1(a)$ for some $a\in [s,t),$ then from the right-continuity of $\alpha_0$ and $\alpha_1$ combining with \eqref{eq14}, we get a contradiction since
$$
\int_a^{b}\alpha_0(r)\zeta(r)\e(\partial_{xx}\Phi_0(r,X_0(r)))^2dr\neq \int_a^b\alpha_\lambda(r)\zeta(r)\e(\partial_{xx}\Phi_\lambda(r,X_\lambda(r)))^2dr
$$ 
for any $b$ sufficiently close to $a.$ Therefore, we reach $\alpha_0=\alpha_1$ on $[s,t]$, which contradicts the definition of $\tau.$ So the inequality \eqref{thm1:eq1} must be strict for all $s\in [\tau',\tau)$ and $x_0,x_1$. This finishes our first part of the argument.
\smallskip

In the second part, we will prove that \eqref{thm1:eq1} is also strict for all $s\in [0,\tau')$ and $x_0,x_1.$ Let $s\in [0,\tau')$ and $x_0,x_1\in\mathbb{R}.$ Take $t=\tau'.$ Recall the representation formula for $\Phi_\lambda(s,u_\lambda^*)$ from \eqref{eq16}.
We observe that from the first part of our proof for $(ii)$, for any $y_0,y_1\in\mathbb{R}$ and $y_\lambda=(1-\lambda)y_0+\lambda y_1,$
\begin{align}\label{eq15}
\Phi_\lambda(t,y_\lambda)<(1-\lambda) \Phi_0(t,y_0)+\lambda\Phi_1(t,y_1).
\end{align}
Therefore, we get the following strict inequality,
\begin{align*}
C_\lambda^{s,t}(u_\lambda^*,x_\lambda)&
<(1-\lambda)C_0^{s,t}(u_\lambda^*,x_0)+\lambda C_1^{s,t}(u_\lambda^*,x_1)
\end{align*}
and using
$$
L_\lambda^{s,t}(u_\lambda^*)=(1-\lambda) L_0^{s,t}(u_\lambda^*)+\lambda L_1^{s,t}(u_\lambda^*)
$$
gives
\begin{align*}
F_\lambda^{s,t}(u_\lambda^*,x_\lambda)<(1-\lambda)F_0^{s,t}(u_\lambda^*,x_0)+\lambda F_1^{s,t}(u_\lambda^*,x_1).
\end{align*}
Thus, since $u_\lambda^*$ is the maximizer for $\Phi_\lambda(s,x_\lambda),$
\begin{align*}
\Phi_\lambda(s,x_\lambda)&=F_\lambda^{s,t}(u_\lambda^*,x_\lambda)\\
&<(1-\lambda)F_0^{s,t}(u_\lambda^*,x_0)+\lambda F_1^{s,t}(u_\lambda^*,x_1)\\
&\leq (1-\lambda) \Phi_0(s,x_0)+\lambda\Phi_1(s,x_1).
\end{align*}
This completes our proof.\end{proof}

\noindent

\end{document}